\documentclass[english,12 pt]{article}
\usepackage[T1]{fontenc}
\usepackage[latin9]{inputenc}
\usepackage{mathrsfs}
\usepackage{amsmath}
\usepackage{amsthm}
\usepackage{amssymb}

\makeatletter

\newcommand{\noun}[1]{\textsc{#1}}

\theoremstyle{plain}
\newtheorem{thm}{\protect\theoremname}
\theoremstyle{definition}
\newtheorem{problem}[thm]{\protect\problemname}
\theoremstyle{remark}
\newtheorem*{claim*}{\protect\claimname}
\theoremstyle{plain}
\newtheorem{lem}[thm]{\protect\lemmaname}

\usepackage[margin=1in]{geometry}
\usepackage{textcomp}
\usepackage{tikz}\usetikzlibrary{fit,positioning}
\usepackage{pdfsync}
\DeclareRobustCommand{\wonwon}{\mathbin{\text{\textwon}}}
\newcommand{\mmesh}{\mathrel{\#}}

\makeatother

\usepackage{babel}
\providecommand{\claimname}{Claim}
\providecommand{\lemmaname}{Lemma}
\providecommand{\problemname}{Problem}
\providecommand{\theoremname}{Theorem}

\begin{document}
\global\long\def\frak#1{\mathfrak{#1}}%
\global\long\def\cal#1{\mathcal{#1}}%
\global\long\def\surr#1{{#1}}%
\global\long\def\t#1{\mathrm{#1}}%
\global\long\def\of#1{{\left(#1\right)}}%
\global\long\def\size#1{\left\Vert #1\right\Vert }%
\global\long\def\bof#1{{\left[#1\right]}}%
\global\long\def\lsup#1#2{{}{}^{#1}{#2}}%
\global\long\def\res{{\restriction\,}}%

\global\long\def\super{\supseteq}%
\global\long\def\subnq{\subsetneqq}%
\global\long\def\all{\forall}%
\global\long\def\exi{\exists}%
\global\long\def\sub{\subseteq}%
\global\long\def\empa{\varnothing}%
\global\long\def\sem{\smallsetminus}%
\global\long\def\ity{\infty}%
\global\long\def\emp{\varnothing}%

\global\long\def\bcup{\bigcup}%
\global\long\def\bcap{\bigcap}%
\global\long\def\and{\wedge}%
\global\long\def\orr{\vee}%
\global\long\def\then{\Rightarrow}%
\global\long\def\tm{\times}%
\global\long\def\isom{\cong}%
\global\long\def\ds{\dots}%
\global\long\def\conc{^{\frown}}%
\global\long\def\nin{\notin}%
\global\long\def\mto{\mapsto}%
\global\long\def\cont{\frak c}%
\global\long\def\sym{\bigtriangleup}%
\global\long\def\clb#1{\overline{#1}}%
\global\long\def\cl{\mathrm{cl}}%
\global\long\def\clm{\mathrm{cl}_{M}}%
\global\long\def\cln{\mathrm{cl}_{N}}%
\global\long\def\clof#1{\mathrm{cl}\left(#1\right)}%

\global\long\def\brp{\mathbb{R}^{+}}%
\global\long\def\br{\mathbb{R}}%
\global\long\def\bz{\mathbb{Z}}%
\global\long\def\bq{\mathbb{Q}}%
\global\long\def\bc{\mathbb{C}}%
\global\long\def\bn{\mathbb{N}}%
\global\long\def\bg{\mathbb{G}}%
\global\long\def\ba{\mathbb{A}}%
\global\long\def\bd{\mathbb{D}}%
\global\long\def\bp{\mathbb{P}}%
\global\long\def\bu{\mathbb{U}}%
\global\long\def\bl{\mathbb{L}}%
\global\long\def\bh{\mathbb{H}}%

\global\long\def\gw{\omega}%
\global\long\def\ga{\alpha}%
\global\long\def\gb{\beta}%
\global\long\def\gd{\delta}%
\global\long\def\ph{\varphi}%
\global\long\def\gga{\gamma}%
\global\long\def\gt{\theta}%
\global\long\def\gz{\zeta}%
\global\long\def\gk{\kappa}%
\global\long\def\gs{\sigma}%
\global\long\def\gl{\lambda}%
\global\long\def\gx{\xi}%
\global\long\def\gp{\pi}%
\global\long\def\eps{\varepsilon}%
\global\long\def\th{\vartheta}%
\global\long\def\del{\Delta}%
\global\long\def\gam{\Gamma}%

\global\long\def\Ord{\text{Ord}}%
\global\long\def\dom{\text{dom}}%
\global\long\def\cof{\text{cf}}%
\global\long\def\ran{\text{ran}}%
\global\long\def\im{\text{im}}%
\global\long\def\ord{\text{On}}%

\global\long\def\abs{\left|\,\right|}%
\global\long\def\mr{\diagdown}%
\global\long\def\mc{\diagup}%
\global\long\def\sl{\cal L}%
\global\long\def\sc{\cal C}%

\global\long\def\Bs{\mathbf{\Sigma}}%
\global\long\def\Bp{\mathbf{\Pi}}%

\global\long\def\cc{\Gamma}%
\global\long\def\dd{}%
\global\long\def\rs#1{{\restriction_{#1}}}%

\global\long\def\abv#1{\left|#1\right|}%
\global\long\def\set#1{\left\{  #1\right\}  }%
\global\long\def\cur#1{\left(#1\right)}%
\global\long\def\cei#1{\left\lceil #1\right\rceil }%
\global\long\def\ang#1{\left\langle #1\right\rangle }%
 
\global\long\def\bra#1{\left[#1\right]}%

\global\long\def\se{\mathscr{E}}%
\global\long\def\sc{\mathscr{C}}%
\global\long\def\sp{\mathscr{P}}%
\global\long\def\sb{\mathscr{B}}%
\global\long\def\sa{\mathscr{A}}%
\global\long\def\ss{\mathscr{S}}%
\global\long\def\sf{\mathscr{F}}%
\global\long\def\sx{\mathscr{X}}%
\global\long\def\sh{\mathscr{H}}%
\global\long\def\su{\mathscr{U}}%
\global\long\def\so{\mathscr{O}}%
\global\long\def\sy{\mathscr{Y}}%
\global\long\def\sm{\mathscr{M}}%
\global\long\def\sn{\mathscr{N}}%
\global\long\def\sd{\mathscr{D}}%
\global\long\def\sl{\mathscr{L}}%
\global\long\def\sg{\mathscr{G}}%
\global\long\def\sk{\mathscr{K}}%
\global\long\def\sv{\mathscr{V}}%
\global\long\def\sw{\mathscr{W}}%
\global\long\def\st{\mathscr{T}}%
\global\long\def\si{\mathscr{I}}%
\global\long\def\sj{\mathscr{J}}%
\global\long\def\sr{\mathscr{R}}%

\global\long\def\df#1{\boldsymbol{#1}}%
\global\long\def\der#1{\overset{\centerdot}{#1}}%
\global\long\def\ocl#1{\left(#1\right]}%
\global\long\def\clo#1{\left[#1\right)}%
\global\long\def\iv#1{\text{Iv}\left(#1\right)}%
\global\long\def\cless#1{\text{cless}\left(#1\right)}%
\global\long\def\dens#1{\text{dens}\left(#1\right)}%
\global\long\def\fr#1#2{\frac{#1}{#2}}%
\global\long\def\mus{\mu^{*}}%
\global\long\def\srs{\sr^{*}}%
\global\long\def\rec#1{\frac{1}{#1}}%
\global\long\def\seqk#1{\left(#1\right)_{k\in\bn}}%
\global\long\def\seqi#1{\left(#1\right)_{i\in\bn}}%
\global\long\def\limk#1{\lim_{k\to\ity}#1_{k}}%
\global\long\def\back{\Leftarrow}%

\global\long\def\lisk#1{k=1,2,\dots,#1}%
\global\long\def\lisi#1{i=1,2,\dots,#1}%
\global\long\def\lisj#1{j=1,2,\dots,#1}%

\global\long\def\cupk#1{\bigcup_{k=1}^{#1}}%
\global\long\def\cupi#1{\bigcup_{i=1}^{#1}}%
\global\long\def\cupj#1{\bigcup_{j=1}^{#1}}%

\global\long\def\capk#1{\bigcap_{k=1}^{#1}}%
\global\long\def\capi#1{\bigcap_{i=1}^{#1}}%
\global\long\def\capj#1{\bigcap_{j=1}^{#1}}%

\global\long\def\seqk#1{\left(#1\right)_{k\in\bn}}%
\global\long\def\seqi#1{\left(#1\right)_{i\in\bn}}%
\global\long\def\seqj#1{\left(#1\right)_{j\in\bn}}%

\global\long\def\limk{\lim_{k\to\ity}}%
\global\long\def\limi{\lim_{i\to\ity}}%
\global\long\def\limj{\lim_{j\to\ity}}%

\global\long\def\sumi#1{\sum_{i=1}^{#1}}%
\global\long\def\sumk#1{\sum_{k=1}^{#1}}%
\global\long\def\sumj#1{\sum_{j=1}^{#1}}%

\global\long\def\qand{\qquad\text{and}\qquad}%
\global\long\def\seg{\left[0,1\right]}%
\global\long\def\bzp{\mathbb{Z}^{+}}%
\global\long\def\ovl#1{\overline{#1}}%
\global\long\def\ome{\Omega}%

\global\long\def\sfl{\mathsf{L}}%
\global\long\def\dof#1{\text{dist}\left(#1\right)}%
\global\long\def\pt{\partial}%
\global\long\def\od{\odot}%
\global\long\def\om{\ominus}%
\global\long\def\nm#1{\left\Vert #1\right\Vert }%

\global\long\def\part#1#2#3{\left\{  \left(#1{}_{1},#2{}_{1}\right),\ds,\left(#1_{#3},#2_{#3}\right)\right\}  }%
\global\long\def\parta#1{\left\{  \left(A_{1},x{}_{1}\right),\ds,\left(A_{#1},x_{#1}\right)\right\}  }%
\global\long\def\partb#1{\left\{  \left(B_{1},y{}_{1}\right),\ds,\left(B_{#1},y_{#1}\right)\right\}  }%

\global\long\def\ac{\text{AC}}%
\global\long\def\gas#1{\alpha^{*}\of{#1}}%
\global\long\def\scivp{\text{SCIVP}}%
\global\long\def\es{\text{ES}}%
\global\long\def\szy{\text{SZ}}%
\global\long\def\conn{\text{Conn}}%
\global\long\def\cn{\mbox{Cn}}%
\global\long\def\szl{\text{SZ}\cur{\mathbb{L}}}%
\global\long\def\won{\wonwon}%
\global\long\def\bb#1{\mathbb{#1}}%
\global\long\def\up#1{#1^{\uparrow}}%
\global\long\def\mesh{\mmesh}%
\global\long\def\adh{\text{adh}}%

\global\long\def\nsub{\nsubseteq}%

\global\long\def\won{\gimel}%

\title{Three problems in convergence theory}
\author{\noun{Jerzy Wojciechowski}\\
West Virginia University}
\date{September 24, 2021}

\maketitle
\thispagestyle{empty}
\begin{abstract}
In this note it is proved that the class of paratopologies is simple
and that under the assumption that the measurable cardinals form a
proper class, the class of hypotopologies is not simple. Moreover,
an example is given of a Hausdorff convergence with idempotent set
adherence (subdiagonal convergence) that is not weakly diagonal.
\end{abstract}

\section{Introduction.}

One way to describe a topological space is to consider the neighborhood
filters of points and the convergence relation between points and
filters defined using the neighborhood filters. Convergence theory
studies this relation in greater generality and considers the topological
convergence only as a special case. The need to study non-topological
convergences was pointed out by Gustave Choquet in his fundamental
paper \cite{Choq-conv}, where he investigates natural convergences
on the family of closed subsets of a topological space and concludes
that some of them are not topological unless the underlying topology
is locally compact. 

The exact collection of axioms required for a convergence space to
satisfy varies in the literature. We follow the definition of Dolecki
in \cite{Dol-init} (see also \cite{DoMy-found,Dol-Royal}). A \emph{convergence}
$\xi$ on a nonempty set $X$ is a relation between the elements of
$X$ and the filters on $X$. Given a filter $\cal F$ on $X$ and
$x\in X$, we write $x\in\lim_{\xi}\cal F$ when $\cur{x,\cal F}\in\xi$
and we require that $\lim_{\xi}\cal F\sub\lim_{\xi}\cal G$ whenever
$\cal F\sub\cal G$ and that $x\in\lim_{\xi}\up{\set x}$ for every
$x\in X$, where $\up{\set x}:=\set{A\sub X:x\in A}$ is the principal
ultrafilter generated by $x$. In particular, any topology on a set
$X$ induces a convergence $\tau$ defined by $x\in\lim_{\tau}\cal F$
if and only if $U\in\cal F$ for every open set $U\sub X$ with $x\in U$.
Any convergence obtained in such a way is called \emph{topological}
or just a \emph{topology}.

Convergences more general than topologies, called pretopologies, had
been already considered by Hausdorff \cite{Hau-raume}, Sierpi\'{n}ski
\cite{Sier-topol} and \v{C}ech \cite{Cech-topol}. A convergence
is a \emph{pretopology} when filters convergent to a point $x$ are
refinements of a single vicinity filter at $x$. However, a breakthrough
was made by Choquet in \cite{Choq-conv} who introduced a still larger
class of \emph{pseudotopologies}, by requiring that $x\in\lim_{\xi}\cal F$
whenever $x\in\lim_{\xi}\cal U$ for every ultrafilter $\cal U$ containing
$\cal F.$

As discovered by Dolecki \cite{Dol-compact}, pseudotopologies arise
in a natural way when we consider the property of compactness. Analogous
considerations for the property of countable compactness lead to paratopologies
and for the Lindelöf property to hypotopologies (see \cite{Dol-compact}).
Those classes of convergences are defined as $\bh$-adherence-determined
convergences for suitably chosen classes $\bh$ of filters.

Let $\xi$ be a convergence on a set $X$ and $\cal H$ be a filter
on $X$. We say that $\cal H$ \emph{adheres} to $x\in X$ (and write
$x\in\adh_{\xi}\cal H$) if there exists a filter $\cal G$ that refines
$\cal H$ with $x\in\lim_{\xi}\cal G$. Given filters $\cal F$ and
$\cal H$, we say that $\cal F$ and $\cal H$ \emph{mesh} (in symbols
$\cal F\mesh\cal H$) iff $F\cap H\neq\emp$ for every $F\in\cal F$
and $H\in\cal H$. Note that if $\xi$ is a convergence on some set
$X$, then $X$ is uniquely determined by $\xi$. We will denote such
$X$ by $\abv{\xi}$. 

Let $\bh$ be a class of filters including all principal filters.
A convergence $\xi$ is $\bh$-\emph{adherence-determined} if $x\in\lim_{\xi}\cal F$
whenever $x\in\t{adh}_{\xi}\cal H$ for each filter $\cal H\in\bh$
such that $\cal H$ is a filter on $\abv{\xi}$ and $\cal H\mesh\cal F$.
Note that a convergence is a pseudotopology if and only if it is $\bb H$-adherence-determined,
where $\bb H$ is the class of all filters and it is a pretopology
when we consider the smallest possible class $\bh$ consisting of
principal filters only.

If we take $\bh$ to be the class of all countably based filters,
then we get the class of \emph{paratopologies}. This class was introduced
by Dolecki \cite{Dol-quotient} to enable a unification (with the
aid of one formula) of various classes of quotient maps (corresponding
to parameters in the formula). In particular, the class of hereditary
quotient maps corresponds to the class of pretopologies and the class
of biquotient maps to pseudotopologies. The class of paratopologies
is obtained in this correspondence when we consider countably biquotient
maps. 

If $\bh$ is the class of all countably complete filters (those that
are closed under countable intersections), then the obtained class
of $\bh$-adherence-determined convergences is the class of \emph{hypotopologies}.
This class of convergences was introduced by Dolecki \cite{Dol-compact}
to enable another unification procedure. 

It turns out that each topology $\tau$ can be represented as the
initial convergence with respect to continuous maps from $\tau$ to
the Sierpi\'{n}ski topology on a set with two elements. The same is
true for pretopologies $\xi$ when we consider maps from $\xi$ to
the Bourdaud pretopology (see Antoine \cite{Ant-etude} and Bourdaud
\cite{Bour-espaces,Bour-cart}). However, as proved by Eva and Robert
Lowen \cite{LoLo-nonsimp}, there is no initially dense pseudotopology
and the same negative result holds for the class of all convergences.
In this paper we will investigate this property for the classes of
paratopologies and hypotopologies.

To formally define the concept of an initial convergence, let's recall
that if $\xi$ and $\eta$ are convergences and $f:\abv{\xi}\to\abv{\eta}$
(where $\abv{\xi}$ and $\abv{\eta}$ are the ground sets for the
convergences $\xi$ and $\eta$, respectively) then $f$ is \emph{continuous}
from $\xi$ to $\eta$ (we write $f\in C\of{\xi,\eta}$) if $f\of x\in\lim_{\eta}f\bof{\cal F}$
for every filter $\cal F$ on $X$ and every $x\in\lim_{\xi}\cal F$.
If $\eta$ is a convergence, $X$ is a set and $f:X\to\abv{\eta}$,
then the relation $f^{-}\eta$ between $X$ and $\bb FX$ (all the
filters on $X$) relating $x$ to $\cal F$ if and only if $f\of x\in\lim_{\xi}f\bof{\cal F}$
is a convergence on $X$. In other words, $\xi:=f^{-}\eta$ is the
coarsest convergence on $X$ that makes $f$ continuous from $\xi$
to $\eta$. 

Let $\Phi$ be a class of convergences. We say that a convergence
$\eta$ is \emph{initially dense} in $\Phi$ iff for each $\xi\in\Phi$
there exists a set $A$ of continuous functions from $\xi$ to $\eta$
such that $\xi=\bigcap_{f\in A}f^{-}\eta$. Note that if $\eta$ is
initially dense in $\Phi$, then $\eta\in\Phi$ and that $\xi=\bigcap_{f\in A}f^{-}\eta$
means that $\xi$ is the coarsest convergence on $\abv{\xi}$ for
which all functions in $A$ are continuous. A class of convergences
is \emph{simple} provided it includes an initially dense convergence. 

As a well known example, recall that a topological space is completely
regular if and only if it is homeomorphic to a subspace of $\br^{X}$
for some set $X$. Using the terminology introduced above, that is
equivalent to saying that the standard topology on the set of real
numbers is initially dense in the class of completely regular topologies.
In particular, the class of completely regular topologies is simple.

We will prove that:
\begin{thm}
\label{thm1}The class of paratopologies is simple.
\end{thm}

To state our result about hypotopologies, we need to recall the concept
of a measurable cardinal (see \cite{Jech}). A \emph{measurable cardinal}
is a cardinal $\gk$ admitting a $\gk$-complete free ultrafilter
(a free ultrafilter closed under intersections of fewer than $\gk$
members).

We are going to show:
\begin{thm}
\label{thm2}Assume that for each cardinal there exists a larger measurable
cardinal. Then the class of hypotopologies is not simple.
\end{thm}

The assumption that measurable cardinals form a proper class is a
very strong set-theoretic assumption. It would be desirable to find
a proof requiring weaker assumptions. In particular, Theorem \ref{thm2}
suggests the following question.
\begin{problem}
Can it be proved in ZFC that the class of hypotopologies is not simple?
\end{problem}

Another property of convergences studied in this paper is diagonality.
Diagonal convergences were defined by Kowalsky \cite{kowal-limes}
(see also \cite{Dol-init,Dol-Royal,DoMy-found}). This property is
important since a topology can be characterized as a diagonal pretopology.
Another way to characterize topologies is to say that a topology is
a pretopology $\xi$ with idempotent set adherence, that is, such
that 
\[
\adh_{\xi}\of{\adh_{\xi}A}=\adh_{\xi}A
\]
for every $A\sub\abv{\xi}$, where
\[
\adh_{\xi}A:=\adh_{\xi}\set{F\sub X:A\sub F}.
\]
The convergences with idempotent set adherence are called \emph{subdiagonal}
by Dolecki \cite{Dol-Royal}. It is true in general that each diagonal
convergence is subdiagonal. 

In \cite{Lo-Co-charact}, Eva Lowen-Colebunders introduced and investigated
convergences $\xi$ such that every filter has a closed adherence,
that is, such that 
\[
\adh_{\xi}\of{\adh_{\xi}\cal F}=\adh_{\xi}\cal F
\]
for every filter $\cal F$ on $\abv{\xi}$. She formulated a condition,
called \emph{weak diagonality}, which is a weakening of the diagonality
property of Kowalsky and proved that a convergence is weakly diagonal
if and only if filters have closed adherences.

Note that each weakly diagonal convergence is subdiagonal. Moreover,
the definition of weak diagonality implies that diagonal convergences
are weakly diagonal. Thus if $\xi$ is a pretopology, then all three
of these concepts are equivalent to $\xi$ being a topology. Example
1.5 in \cite{Lo-Co-charact} shows that there exists a subdiagonal
convergence which is not weakly diagonal. In that example, however,
the convergence is not Hausdorff. We say that a convergence is \emph{Hausdorff}
provided that each filter has at most one limit.

We will prove the following result.
\begin{thm}
\label{thm4}There exists a Hausdorff subdiagonal convergence that
is not weakly diagonal.
\end{thm}

\section{Proof of Theorem \ref{thm1}.}

Let $X=\gw\cup\set{\ity_{0},\ity_{1},\ity_{2}}$, where $\ity_{0},\ity_{1},\ity_{2}$
are distinct and do not belong to $\gw$. Let $\won$ be the convergence
on $X$ such that a filter $\cal F$ converges to $x$ iff either
$x\neq\ity_{2}$ or $K\cup\set{\ity_{1},\ity_{2}}\in\cal F$, for
some finite $K\sub\gw$. Note that, in particular, $\ity_{2}\in\lim_{\gimel}\cal F$
if and only if there is finite $F\in\cal F$ with $\ity_{0}\nin F$.

We verify that $\won$ is a paratopology. Indeed, assume $x\in X\sem\lim_{\gimel}\cal F$.
Then $x=\ity_{2}$ and $\ity_{0}\in F$ for every finite $F\in\cal F$.
We want a countably based filter $\cal H$ such that $\cal F\mesh\cal H$
and $x\nin\adh_{\gimel}\cal H.$ If $\ity_{0}\in\bigcap\cal F$, then
$\cal H:=\up{\set{\ity_{0}}}$ satisfies the requirements. Otherwise,
all sets in $\cal F$ are infinite and we can use the cofinite filter
on $X$ as $\cal H$.

We will show that $\won$ is initially dense in the class $\Phi$
of paratopologies. Let $\eta$ be a paratopology on a set $Y$. We
need to find a set $C$ of continuous functions from $\eta$ to $\gimel$
such that $\eta=\bigcap_{f\in C}f^{-}\gimel$.

Let $\mathbf{B}$ be the collection of all countable families
\[
\cal B=\set{B_{0},B_{1},\ds},
\]
of subsets of $Y$, where $B_{0}\super B_{1}\super B_{2}\super\ds$.
For each $\cal B\in\mathbf{B}$, let $f_{\cal B}:Y\to X$ be defined
by
\[
f_{\cal B}\of y=\begin{cases}
n & \text{if }y\in B_{n}\sem B_{n+1}\text{ for some }n\in\gw\\
\ity_{0} & \text{if }y\in B_{n}\text{ for every }n\in\gw\\
\ity_{1} & \text{if }y\in\t{adh}_{\eta}\cal H\sem\bigcup_{n\in\gw}B_{n},\text{ where }\cal H\text{ is the filter generated by \ensuremath{\cal B}}\\
\ity_{2} & \text{otherwise}.
\end{cases}
\]
Let 
\[
C:=\set{f_{\cal B}:\cal B\in{\bf B}}.
\]
It remains to verify that $\eta=\bigcap_{f\in C}f^{-}\gimel$. Let
$\cal F$ be a filter on $Y$ and $y\in Y$. It suffices to show that
$y\in\lim_{\eta}\cal F$ if and only if $f\of y\in\lim_{\won}f\bof{\cal F}$
for every $f\in C$.
\begin{claim*}
Assume that $y\in\lim_{\eta}\cal F$. Then $f\of y\in\lim_{\won}f\bof{\cal F}$
for every $f\in C.$
\end{claim*}
\begin{proof}
Let $f=f_{\cal B}\in C$ with
\[
\cal B=\set{B_{0},B_{1},\ds}\in\mathbf{B},
\]
where $B_{0}\super B_{1}\super B_{2}\super\ds$, and let $\cal H$
be the filter generated by $\cal B$. If $\cal F$ and $\cal H$ do
not mesh, then there is $F\in\cal F$ and $n\in\gw$ such that $F\cap B_{m}=\emp$
for every $m>n$. Taking $K:=\set{k\in\gw:k\le n}$ we have $K\cup\set{\ity_{1},\ity_{2}}\in f\bof{\cal F}$
so $f\of y\in\lim_{\won}f\bof{\cal F}$. If $\cal F\mesh\cal H$,
then $y\in\t{adh}_{\won}\cal H$ so 
\[
f\of y\in X\sem\set{\ity_{2}}\sub\t{lim}_{\won}f\bof{\cal F}.
\]
\end{proof}
\begin{claim*}
Assume that $y\nin\lim_{\eta}\cal F$. Then $f\of y\nin\lim_{\won}f\bof{\cal F}$
for some $f\in C.$
\end{claim*}
\begin{proof}
Let $\cal H$ be a countably based filter such that $\cal F\mesh\cal H$
and $y\nin\t{adh}_{\eta}\cal H$. Since $y\nin\t{adh}_{\eta}\cal H$,
there is $H\in\cal H$ with $y\nin H$. Let 
\[
\cal B=\set{B_{0},B_{1},\ds}\in\mathbf{B},
\]
with $B_{0}\super B_{1}\super B_{2}\super\ds$, be a base for $\cal H$
with $B_{0}\sub H$. Since $y\nin B_{0}\cup\t{adh}_{\eta}\cal H$,
it follows that $f_{\cal B}\of y=\ity_{2}$.

It remains to show that $\ity_{2}\nin\lim_{\won}f_{\cal B}\bof{\cal F}$.
Suppose, for a contradiction, that $\ity_{2}\in\lim_{\won}f_{\cal B}\bof{\cal F}$.
Then there is $F\in\cal F$ such that $\ity_{0}\nin f_{\cal B}\of F$
and $f_{\cal B}\of F\cap\gw$ is finite.

Let $B:=\bigcap_{n\in\bn}B_{n}$. Since $\ity_{0}\nin f_{\cal B}\of F$,
it follows that $F\cap B=\empa$. Since $f_{\cal B}\of F\cap\gw$
is finite, there is $n\in\gw$ such that $m\nin f_{\cal B}\of F$
for every $m\ge n$. Since $\cal F\mesh\cal H$, we have $F\cap B_{n}\neq\emp$.
As $F\cap B=\empa$, it follows that there is $m\ge n$ with $F\cap\cur{B_{m}\sem B_{m+1}}\neq\emp$.
Since $m\in f_{\cal B}\of F$ for such $m$, we get a contradiction.
\end{proof}

\section{Proof of Theorem \ref{thm2}.}

We modify the argument from \cite{LoLo-nonsimp}. Assume that for
each cardinal there exists a larger measurable cardinal. A filter
$\cal F$ on a set $X$ is \emph{uniform} iff each member of $\cal F$
has the same cardinality as $X$.
\begin{lem}
\label{lem4}Let $X$ and $Y$ be sets such that $X$ is uncountable
and $\t{card}\,X>\t{card}\,Y$. Then for each uniform countably complete
ultrafilter $\cal U$ on $X$ and each $f:X\to Y$ there exists a
uniform countably complete ultrafilter $\cal W$ on $X$ such that
$\cal W\neq\cal U$ and $f\bof{\cal U}=f\bof{\cal W}$.
\end{lem}

\begin{proof}
Let $\cal P:=\set{f^{-}\of y:y\in f\of X}$ with $\cal P_{0}:=\set{A\in\cal P:\t{card}\,A\le\aleph_{0}}$
and $\cal P_{1}:=\cal P\sem\cal P_{0}$. Note that $\t{card}\,P_{0}<\t{card}\,X$,
where
\[
P_{0}:=\bigcup_{P\in\cal P_{0}}P
\]
and so $\cal P_{1}$ is not empty. 

For each $P\in\cal P_{1}$, let $P=A_{P}\cup B_{P}$ with $A_{P}\cap B_{P}=\emp$
and $\t{card}\,A_{P}=\t{card}\,B_{P}$. Let
\[
A:=P_{0}\cup\bigcup_{P\in\cal P_{1}}A_{P}\qquad\text{and\ensuremath{\qquad}}B:=P_{0}\cup\bigcup_{P\in\cal P_{1}}B_{P}.
\]
Then $A\cup B=X\in\cal U$ and $A\cap B=P_{0}\nin\cal U$ since $\cal U$
is uniform. Thus exactly one of $A$ and $B$ belongs to $\cal U$.
Let $h:A\to B$ be a bijection such that $h\of x=x$ for each $x\in P_{0}$
and $h$ maps $A_{P}$ onto $B_{P}$ for each $P\in\cal P_{1}$. Then
$\cal W:=h\bof{\cal U}$ is a countably complete ultrafilter on $X$
with $\cal U\neq\cal W$ and $f\bof{\cal U}=f\bof{\cal W}$ as required.
\end{proof}
\begin{lem}
Let $X$ be a set of measurable cardinality. Then there exists a uniform
countably complete ultrafilter on $X$.
\end{lem}

\begin{proof}
Let $\gk:=\t{card}\,X$. Since $\gk$ is measurable, there exists
a free $\gk$-complete ultrafilter $\cal U$ on $X$. Since $\cal U$
is free, $X\sem\set x\in\cal U$ for every $x\in X$. Since $\cal U$
is $\gk$-complete, if $A\sub X$ has cardinality smaller than $\gk$,
then 
\[
X\sem A=\bigcap_{x\in A}\cur{X\sem\set x}\in\cal U,
\]
so $A\nin\cal U$. Thus $\cal U$ is uniform.
\end{proof}
\begin{lem}
\label{lem6}Let $\cal U$ be a uniform countably complete ultrafilter
on an uncountable set $X$ and $x_{\ity}\in X$. Define a convergence
$\xi=\xi\of{\cal U,x_{\ity}}$ on $X$ by $x\in\lim_{\xi}\cal F$
iff $\cal F=\up{\set x}$ for $x\in X\sem\set{x_{\ity}}$ and $x_{\ity}\in\lim_{\xi}\cal F$
iff $\bigcap\cal F\sub\set{x_{\ity}}$ and $\cal F\nsub\cal U$. Then
$\xi$ is a hypotopology.
\end{lem}

\begin{proof}
Let $\cal F$ be a filter on $X$ and $x\in X\sem\lim_{\xi}\cal F$.
We need to find a countably complete filter $\cal H$ on $X$ such
that $\cal H\mesh\cal F$ but $x\nin\adh_{\xi}\cal H$. If $x\neq x_{\ity}$,
then $\cal F\neq\up{\set x}$. Taking any $A\in\up{\set x}\sem\cal F$
makes $\cal H:=\up{\cur{X\sem A}}$ to be as required.

Assume $x=x_{\ity}$. If there is $y\in\bigcap\cal F$ with $y\neq x$,
then $\cal H:=\up{\set y}$ satisfies the requirements. If $\bigcap\cal F\sub\set x$,
then $\cal F\sub\cal U$ and $\cal H:=\cal U$ satisfies the requirements.
\end{proof}

\paragraph{Proof of Theorem \ref{thm2}.}

Let $\eta$ be a convergence on a set $Y$. Let $X$ be such that
$\t{card}\,X$ is a measurable cardinal with $\t{card}\,X>\t{card}\,Y$
and let $\cal U$ be a uniform countably complete ultrafilter on $X$.
Let $\xi:=\xi\of{\cal U,x_{\ity}}$ be the hypotopology on $X$ as
in Lemma \ref{lem6} for some $x_{\ity}\in X$. By Lemma \ref{lem4},
for every map $f\in Y^{X}$ there is a uniform countably complete
ultrafilter $\cal W_{f}$ on $X$ with $\cal W_{f}\neq\cal U$ and
$f\bof{\cal U}=f\bof{\cal W_{f}}$. Then $\lim_{\xi}\cal U=\emp$
and $x_{\ity}\in\lim_{\xi}\cal W_{f}$ for each $f\in Y^{X}$. Thus
\[
\xi\neq\bigcap_{f\in\sc\cur{\xi,\tau}}f^{-}\eta.
\]

\section{Proof of Theorem \ref{thm4}.}

Let $X_{n}$ be a countably infinite set for each $n\in\gw$ with
$X_{n}\cap X_{m}=\empa$ whenever $n\neq m$. Let $y_{0},y_{1},\ds$
be distinct with $\set{y_{0},y_{1},\ds}\cap X_{n}=\empa$ for each
$n\in\gw$. Let 
\[
z\nin\set{y_{0},y_{1},\ds}\cup\bigcup_{n\in\gw}X_{n}
\]
and
\[
X:=\bigcup_{n\in\gw}X_{n}\cup\set{y_{0},y_{1},\ds}\cup\set z.
\]
We define a Hausdorff pseudotopology $\xi$ on $X$ as follows. If
$x\in\bigcup_{n\in\gw}X_{n}$, then the principal ultrafilter $\up{\set x}$
is the only filter on $X$ that converge to $x$. A free ultrafilter
$\cal U$ on $X$ converges to $y_{n}$ for $n\in\gw$ iff $\cal U$
refines the cofinite filter on $X_{n}$. A free ultrafilter $\cal U$
on $X$ converges to $z$ iff $\cal U$ refines the cofinite filter
on $\set{y_{0},y_{1},\ds}$ or there exists a sequence $x_{0},x_{1},\ds$
with $x_{n}\in X_{n}$ for each $n\in\gw$ and $\cal U$ refines the
cofinite filter on $\set{x_{0},x_{1},\ds}$.
\begin{claim*}
The convergence $\xi$ is subdiagonal.
\end{claim*}
\begin{proof}
Let $A\sub X$ and $x\in X$ be such that $x\in\t{adh_{\xi}}\cur{\t{adh}_{\xi}A}$.
If $x\in\bigcup_{n\in\gw}X_{n}$, then $x\in\t{adh}_{\xi}A$. 

Assume that $x=y_{n}$ for some $n\in\gw.$ Since $x\in\t{adh_{\xi}}\cur{\t{adh}_{\xi}A}$,
it follows that $x\in\adh_{\xi}A$ or $X_{n}\cap\adh_{\xi}A$ is infinite.
Note that 
\[
X_{n}\cap\adh_{\xi}A=X_{n}\cap A,
\]
and that if $X_{n}\cap A$ is infinite, then $x\in\adh_{\xi}A$. Therefore
$x\in\adh_{\xi}A$.

Assume that $x=z$ and $x\nin\t{adh}_{\xi}A$. Then the set $\set{n\in\gw:A\cap X_{n}\neq\empa}$
is finite, implying that the set $\set{n\in\gw:y_{n}\in\t{adh}_{\xi}A}$
is also finite which contradicts $x\in\t{adh_{\xi}}\cur{\t{adh}_{\xi}A}$.
\end{proof}
\begin{claim*}
The convergence $\xi$ is not weakly diagonal.
\end{claim*}
\begin{proof}
Let $\cal F$ be the free filter on $X$ such that for $A\sub X$
we have $A\in\cal F$ iff $X_{n}\sem A$ finite for every $n\in\gw$.
We show that 
\[
z\in\t{adh}_{\xi}\cur{\t{adh}_{\xi}\cal F}\sem\t{adh}_{\xi}\cal F.
\]
Indeed, $z\in\t{adh}_{\xi}\cur{\t{adh}_{\xi}\cal F}$ since $y_{n}\in\t{adh_{\xi}\cal F}$
for each $n\in\gw$ and $z\in\t{adh}_{\xi}\set{y_{0},y_{1},\ds}$.
However, $z\nin\t{adh_{\xi}\cal F}$ since for every ultrafilter $\cal U$
that converges to $z$ either $\set{y_{0},y_{1},\ds}\in\cal U$ or
there is a sequence $x_{0},x_{1},\ds$ with $x_{n}\in X_{n}$ for
every $n\in\gw$ and $\set{x_{0},x_{1},\ds}\in\cal U$. If the former
holds, then $\cal U$ does not refine $\cal F$. If the latter holds,
then $\cal U$ does not refine $\cal F$ either since 
\[
\bigcup_{n\in\gw}\cur{X_{n}\sem\set{x_{n}}}\in\cal F\sem\cal U.
\]
Thus $\xi$ is not weakly diagonal.
\end{proof}

\section*{Acknowledgment}

I am very grateful to Szymon Dolecki (Institut de Mathématiques de
Bourgogne) for making me interested in convergence theory. I am also
thankful to both Szymon and Frédéric Mynard (New Jersey City University)
for very helpful remarks on the first draft of this paper and to the
anonymous referee whose remarks allowed me to make some necessary
corrections.


\begin{thebibliography}{10}
\bibitem{Ant-etude}\noun{P.~Antoine}, \emph{Etude élémentaire des
catégories d\textquoteright ensembles structurés}, Bull.~Soc. Math.
Belgique \textbf{18} (1966), 142\textendash 164.

\bibitem{Bour-espaces}\noun{G.~Bourdaud}, \emph{Espaces d\textquoteright Antoine
et semi-espaces d\textquoteright Antoine}, Cahiers de topologie et
géométrie différentielle catégoriques \textbf{16} (1975), 107\textendash 133.

\bibitem{Bour-cart}\noun{G.~Bourdaud}, \emph{Some cartesian closed
topological categories of convergence spaces}, Categorical Topology,
93\textendash 108, Lecture Notes in Math 540, Springer-Verlag, 1975.

\bibitem{Cech-topol}\noun{E.~\v{C}ech}, \emph{Topological Spaces},
Wiley, 1966.

\bibitem{Choq-conv}\noun{G.~Choquet}, \emph{Convergences}, Ann.
Univ. Grenoble \textbf{23} (1947\textendash 48), 55\textendash 112.

\bibitem{Dol-quotient}\noun{S.~Dolecki}, \emph{Convergence-theoretic
approach to quotient quest}, Topology Appl.~\textbf{73} (1996), 1\textendash 21.

\bibitem{Dol-compact}\noun{S.~Dolecki}, \emph{Convergence-theoretic
characterization of compactness}, Topology Appl.~\textbf{125} (2002),
393\textendash 417.

\bibitem{Dol-erratum}\noun{S.~Dolecki}, \emph{Erratum to ``Convergence-theoretic
characterization of compactness}'', Topology Appl.~\textbf{154}
(2007), 1216\textendash 1217.

\bibitem{Dol-init}\noun{S.~Dolecki}, \emph{An initiation into Convergence
Theory}, Beyond Topology, 115\textendash 161, F.~Mynard and E.~Pearl,
eds, Contemporary Mathematics 486, AMS, Providence, 2009.

\bibitem{Dol-Royal}\noun{S.~Dolecki}, \emph{A Royal Road to Topology:
Convergence of Filters}, World Scientific, 2022, to appear.

\bibitem{DoMy-found}\noun{S.~Dolecki} and \noun{F.~Mynard}, \emph{Convergence
Foundations of Topology}, World Scientific, 2016.

\bibitem{Hau-raume}\noun{F.~Hausdorff}, \emph{Gestufte Räume: Fund.~Math.}~\textbf{25}
(1935), 486\textendash 506.

\bibitem{Jech}\noun{T.~Jech}, \emph{Set Theory}, Springer, 2003.

\bibitem{kowal-limes}\noun{H.~Kowalsky}, \emph{Limesräume und Komplettierung:
Fath.~Nachr.}~\textbf{12} (1954), 301\textendash 340.

\bibitem{Lo-Co-charact}\noun{E.~Lowen-Colebunders}, \emph{An internal
and an external characterization of convergence spaces in which adherences
of filters are closed}, Proc. Amer. Math. Soc. \textbf{72} (1978),
205\textendash 210.

\bibitem{LoLo-nonsimp}\noun{E.~Lowen} and \noun{R.~Lowen}, \emph{On
the nonsimplicity of some convergence categories}, Proc. Amer. Math.
Soc. \textbf{105} (2), (1989) 305\textendash 308.

\bibitem{Sier-topol}\noun{W.~Sierpi\'{n}ski}, \emph{General Topology},
Univ.~Toronto, 1934.
\end{thebibliography}
\end{document}